\documentclass{amsart}
\usepackage{amssymb,amsthm,amscd}
\usepackage[pdftex]{graphics}
\include{diagrams}
\usepackage{hyperref}

\newcommand{\Q}{{\mathbb Q}}
\newcommand{\Z}{{\mathbb Z}}

\newcommand{\fa}{{\mathfrak a}}
\newcommand{\fb}{{\mathfrak b}}

\newcommand{\fc}{{\mathfrak c}}

\newcommand{\fm}{{\mathfrak m}}
\newcommand{\fp}{{\mathfrak p}}
\newcommand{\fq}{{\mathfrak q}}
\newcommand{\fP}{{\mathfrak P}}

\newcommand{\cO}{{\mathcal O}}
\newcommand{\cok}{{\operatorname{cok}\,}}

\newcommand{\Am}{{\operatorname{Am}}}
\newcommand{\st}{{\operatorname{st}}}

\newcommand{\Ram}{\mbox{\rm Ram}}
\newcommand{\Gal}{\mbox{\rm Gal}\,}

\newcommand{\Cl}{\operatorname{Cl}}

\newcommand{\im}{\operatorname{im}}
\newcommand{\Sel}{\operatorname{Sel}}
\newcommand{\eps}{\varepsilon}

\newcommand{\impl}{\Longrightarrow}

\newcommand{\llra}{\Longleftrightarrow}

\newcommand{\la}{\langle}
\newcommand{\ra}{\rangle}

\newcounter{lemmacount}[section]

\newtheorem{thm}[lemmacount]{Theorem}
\newtheorem{prop}[lemmacount]{Proposition}
\newtheorem{lem}[lemmacount]{Lemma}
\newtheorem{cor}[lemmacount]{Corollary}
\newtheorem{fact}[lemmacount]{Fact}
{}
{\bf}{\it}
{\bf}{}

\title{Scholz's Reciprocity Law in Number Fields}
\author{Franz Lemmermeyer}
\address{M\"orikeweg 1, 73489 Jagstzell, Germany}
\email{hb3@ix.urz.uni-heidelberg.de}

\begin{document}

\begin{abstract}
  In this article we prove a reciprocity law in number fields with odd
  class number that specializes to Scholz's reciprocity law over the
  rationals.
\end{abstract}

\maketitle

In this article we will generalize Scholz's reciprocity law
to number fields with odd class number (there are similar results
in general number fields, with unit groups replaced by Selmer groups,
but they cannot be formulated as smoothly as in the case of
odd class numbers).

\section{The Classical Version}
Before we explain the content of the generalization of Scholz's
reciprocity law, let us first discuss the classical version. 
Recall that the quadratic residue symbol $[\alpha/\fp]$ in a number
field takes the values $\pm 1$ whenever $\fp \nmid (\alpha)$, and is
determined by Euler's congruence
$[\alpha/\fp] \equiv \alpha^{(N\fp-1)/2} \bmod \fp$. If
$\alpha \equiv \xi^2 \bmod \fp$ is a quadratic residue, then we set
$[\alpha/\fp]_4 = [\xi/\fp]$.

Now let $p \equiv q \equiv 1 \bmod 4$ be positive primes. 
Let $\eps_p$ and $\eps_q$ denote the fundamental units of 
$\Q(\sqrt{p}\,)$ and $\Q(\sqrt{q}\,)$, respectively. If we
assume that $(p/q) = (q/p) = +1$, then $(p) = \fp\fp'$ and 
$(q) = \fq\fq'$ split in $\Q(\sqrt{q}\,)$ and $\Q(\sqrt{p}\,)$, 
respectively. It is easily checked that the quadratic residue 
symbol $[\eps_p/\fq]$ in $\Q(\sqrt{p}\,)$ does not depend on 
the choice of $\fq$, hence we may denote it by $(\eps_p/q)$.
The part of Scholz's theorem that we will generalize here is

\begin{thm}\label{TSRC}
  Let $k$ be a quadratic number field in which exactly the
  odd prime numbers $p$ and $q$ ramify. Let $h$ and $\eps$ denote
  the class number and the fundamental unit of $k$, respectively.
  Then $p \equiv q \bmod 4$, and there are the following cases:
  \begin{enumerate}
  \item[(a)] $p \equiv q \equiv 3 \bmod 4$: then $h$ is odd,
    and $N\eps = +1$.
  \item[(b)] $p \equiv q \equiv 1 \bmod 4$ and $(p/q) = -1$:
    Then $h \equiv 2 \bmod 4$ and $N\eps = -1$.
  \item[(c)] $p \equiv q \equiv 1 \bmod 4$ and $(p/q) = +1$:
    In this case, we have the reciprocity law
    $$ \Big(\frac{\eps_p}q\Big) = \Big(\frac{\eps_q}p\Big) $$
    due to Sch\"onemann and Scholz. In addition we have the following results:
    \begin{itemize}
    \item If $(\eps_p/q) = -1$, then  $h \equiv 2 \bmod 4$ and $N\eps = +1$.
    \item If $(\eps_p/q) = +1$, then  $h \equiv 0 \bmod 4$.
    \end{itemize}
  \end{enumerate}        
\end{thm}

For a proof and the history of the reciprocity law we refer to
\cite{LemRL}; Scholz's proof can be found in \cite{Sch}.

\section{Hilbert's First Supplementary Law}
Let $F$ be a number fields with odd class number, and denote its 
unit group $\cO_F^\times$ by $E = E_F$. For an ideal $\fa$ with 
odd norm we put $(E/\fa) = +1$ if and only if $(\eps/\fa) = +1$ 
for all units $\eps \in E$, where $(\eps/\fa)$ is the usual quadratic 
Jacobi symbol in $F$; if there is a unit with $(\eps/\fa) = -1$,
then we set $(E/\fa) = -1$. The ideal $\fa$ is called primary if 
and only if $(E/\fa) = +1$. 

We will say that an element $\alpha \in \cO_F$ with odd norm is 
$2$-primary if $\alpha \equiv \xi^2 \bmod 4$ for some $\xi \in \cO_F$,
and we will call it primary if it is $2$-primary and totally positive.
An element $\alpha$ is $2$-primary if and only if every ideal above 
$2$ is unramified in $F(\sqrt{\alpha}\,)/F$; it is primary if
in addition there is no ramification at the infinite primes.
The connection between primary ideals and primary elements is given 
by the first supplementary law of the quadratic reciprocity law in
number fields in the form given by Hilbert:

\begin{thm}[The First Supplementary Law]\label{Hil1SKa} 
Let $F$ be a number field with odd class number $h$, and suppose 
that $\fa$ is a squarefree ideal with odd norm. Then the following
assertions are equivalent:
\begin{enumerate}
\item[i)] $\fa$ is primary, i.e. $(\eps/\fa) = +1$ for all units
        $\eps \in E_F$;
\item[ii)] there exists an integral ideal $\fc$ and a primary 
       $\beta \in \cO_F$ such that $\fa\fc^2 = (\beta)$;
\item[iii)] there exists a primary $\alpha \in \cO_F$ such that 
       $\fa^h = (\alpha)$.
\item[iv)] there exists a quadratic extension $K/F$ ramified
          exactly at the primes $\fp \mid \fa$;
\end{enumerate}
Moreover, the extension $K/F$ in {\em iv)} is uniquely determined by $\fa$.
\end{thm}

The heart of this theorem is the equivalence i) $\llra$ ii);
this was proved in \cite[Thm.~6.5]{LSel}; observe that since
$F$ has odd class number, the Selmer group $\Sel(F)$ is 
the unit group $E_F$ in this case.

We now prove that \ref{Hil1SKa}. ii), iii), and  iv) are equivalent.
\begin{itemize}
\item ii) $\impl$ iv): Consider $K = F(\sqrt{\beta}\,)$. Since
          $\beta$ is primary, $K/F$ is unramified at the primes 
          $2\infty$, and since $\fa$ is squarefree, $K/F$ is 
          ramified exactly at the primes $\fp \mid \fa$.
\item iv) $\impl$ ii): If $K/F$ is any quadratic extension ramified 
          exactly at the primes $\fp \mid \fa$, write 
          $K = F(\sqrt{\beta}\,)$ for some $\beta$ coprime to $2$. 
          Since no prime above $2\infty$ ramifies $\beta$ is 
          primary. Since exactly the primes $\fp \mid \fa$ ramify, 
          we must have $(\beta) = \fa\fc^2$ for some ideal $\fc$. 
\item iii) $\impl$ ii): put $\fa = \fp^{(h-1)/2}$.
\item ii) $\impl$ iii): But if $\fa \fc^2 = (\beta)$
          for some primary $\beta$, then $\fa^h \fc^{2h} = (\beta)^h$. 
          Write $\fc^h = (\gamma)$; then $\fa^h = (\alpha)$ for 
          $\alpha = \beta^h/\gamma^2$, and since $\beta$ is primary,
          so is $\alpha$.
\end{itemize}
The uniqueness of the extension $K/F$ is clear: if $K'/F$
also has these properties, the compositum $KK'/F$ contains
a quadratic subextension $L/F$ which is easily seen to be
unramified everywhere, contradicting our assumption that
the class number of $F$ be odd.

The following analog of Theorem \ref{Hil1SKa} can be proved
in a similar way:

\begin{thm}[The strict version of the First Supplementary Law]
\label{Hil1p} 
Let $F$ be a number field with odd class number $h^+$ in the strict 
sense, and suppose that $\fa$ is a squarefree ideal with odd norm. 
Then the following assertions are equivalent:
\begin{enumerate}
\item[i)] $(\eps/\fa) = +1$ for all units $\eps \in E_F^+$;
\item[ii)] there exists an integral ideal $\fc$ and a $2$-primary
            $\beta \in \cO_F$ such that $\fa\fc^2 = (\beta)$.
\item[iii)] There exists a $2$-primary $\alpha \in \cO_F$ 
            such that $\fa^{h^+} = (\alpha)$.
\item[iv)] There exists a quadratic extension $K/F$ ramified
           exactly at the primes $\fp \mid \fa$, and possibly
           at some infinite primes.
\end{enumerate}
Moreover, the extension $K/F$ in {\em iv)} is uniquely determined by $\fa$.
\end{thm}

We will also need the ambiguous class number formula. Recall that an ideal
  class $c$ in a cyclic extension $K/F$ is called ambiguous if it is
  fixed by $\Gal(K/F)$, and it is called strongly ambiguous if $c = [\fa]$
  for an ideal $\fa$ fixed by $\Gal(K/F)$. The ambiguous and strongly
  ambiguous ideal classes form groups $\Am(K/F)$ and $\Am_\st(K/F)$, and
  their cardinalities are given by the formulas (see \cite{LemA})
  \begin{equation}\label{ECam}
    \# \Am     = h_F \frac{\prod e(\fp)}{2(E_F:E_F \cap NK^\times)}, \quad
    \# \Am_\st = h_F \frac{\prod e(\fp)}{2(E_F:NE_K)}.
  \end{equation}
  where $e(\fp)$ is the ramification index of the prime $\fp$ in $F$
  (the product is over all primes $\fp$, finite and infinite).

\section{Scholz's Theorem}

Let $F$ be a number field with odd class number and $k/F$ a quadratic
extension with exactly two ramified primes: $\Ram(k/F) = \{\fp_1, \fp_2\}$.

\begin{fact}
  We have  $(E_F/\fp_1) =  (E_F/\fp_2)$.
\end{fact}

\begin{proof}
  Since in $k$ only $\fp_1$ and $\fp_2$ ramify, the product $\fp_1\fp_2$
  must be primary in $F$, i.e., $(E_F/\fp_1\fp_2) = +1$.
  This implies the claim.
\end{proof}

The main theorem we will prove in this article the following:

\begin{thm}\label{Tmain}
  Let $F$ be a number field with odd class number, and let $k/F$ be a
  quadratic extension with exactly two ramified prime ideals $\fp_1$ and
  $\fp_2$. Let $h$ denote the class number of $k$.
  \begin{enumerate}
  \item If $\fp_1$ and $\fp_2$ are not primary, i.e., if
    $(E_F/\fp_1) = (E_F/\fp_2) = -1$, then $h$ is odd and $(E_F:NE_k) = 2$.
  \item If $\fp_1$ and $\fp_2$ are primary and $(\fp_1/\fp_2) = -1$,
    then $h \equiv 2 \bmod 4$ and $E_F = NE_k$.
    \item If $\fp_1$ and $\fp_2$ are primary and $(\fp_1/\fp_2) = +1$,
      then we have the reciprocity law $(E_1/\fp_2) = (E_2/\fp_1)$. In
      addition, we have the following results:
    \item If $(E_1/\fp_2) = -1$, then $h \equiv 2 \bmod 4$ and
      $(E_F:NE_k) = 2$.
    \item If $(E_1/\fp_2) = +1$, then $h \equiv 0 \bmod 4$.
  \end{enumerate}  
\end{thm}

We begin by proving claim (1). This will follow from the following

\begin{fact}
  The class number $h(k)$ is odd if and only if $(E_F/\fp_1) = -1$.
  In this case we have $(E_F:NE_k) = 2$, and in fact
  $(E_F:E_F \cap Nk^\times) = 2$.
\end{fact}

\begin{proof}
  If $(E_F/\fp_1) = +1$, then $\fp_1$ is primary, hence $F(\sqrt{\pi_1}\,)$
  is a quadratic extension in which only $\fp_1$ is ramified. By Abhyankar's
  Lemma, the extension $k(\sqrt{\pi_1}\,)/k$ is unramified, and this
  shows that $h(k)$ is even.

  Conversely, if $(E_F/\fp_1) = -1$, there does not exist a quadratic
  extension of $F$ ramified exactly at $\fp$, hence there is no unramified
  quadratic extension of $k$, and therefore $k$ has odd class number.

  We now invoke the ambiguous class number formula. If the class number
  of $k$ is odd, then the ambiguous class number formula tells us that
  $$ \Am(k/F) = h_k \frac{2}{(E_F:E_F:E_F \cap NK^\times))}
      \quad \text{and} \quad
     \Am_\st(k/F) = h_k \frac{2}{(E_F:NE_k)} $$
  are both odd, which implies that
 $(E_F:E_F:E_F \cap NK^\times)) = (E_F:NE_k) = 2$.     
\end{proof}

Next we give the 
\begin{proof}[Proof of (2)]
  Since $\fp_1$ is primary, we can write $\fp_1^{h(F)} = (\pi_1)$ for some
  primary $\pi_1$; here $h(F)$ is the (odd) class number of $F$. Then
  $K = k(\sqrt{\pi_1}\,)$ is unramified over $k$, and $F_1 = F(\sqrt{\pi_1}\,)$
  has odd class number since only $\fp_1$ ramifies in $F_1/F$. Since
  $(\fp_1/\fp_2) = -1$, $\fp_2$ remains inert in $F_1/F$; thus exactly
  one prime ideal (namely $\fp_2$) ramifies in $K/F_1$, and the
  ambiguous class number formula implies that $h_K$ is odd. This in turn
  implies that $h_k \equiv 2 \bmod 4$.
  
  We now claim that the prime ideal $\fP_1$ in $k$ above $\fp_1$ has
  the property that $\fP_1^{h/2}$ is not principal; this implies that
  $\fp_1$ is an ambiguous ideal of order $2$ in $\Cl(k)$, hence
  $\Am_\st(k/F) = 2$, and the ambiguous ideal class then implies
  that $E_F = NE_k$. In fact, if $\fP_1^{h/2}$ is principal, then it splits
  in the Hilbert $2$-class field $K/k$; but this implies that $\fp_1$
  splits in $k_2/F$, hence $(\fp_2/\fp_1) = +1$ contradicting our
  assumption.
\end{proof}

From now on we assume in addition that $(\fp_1/\fp_2) = (\fp_2/\fp_1) = +1$,
i.e., that $\fp_1 = \fP_1\fP_1'$ splits in $k_2/F$ and that
$\fp_2 = \fP_2\fP_2'$ splits in $k_1/F$. We begin by showing that the
symbol $(E_1/\fP_1)$ does not depend on the choice of the conjugate:

\begin{fact}
  We have $(E_2/\fP_1) = (E_2/\fP_1')$, where $E_2$ is the unit group of $k_2$.
\end{fact}

\begin{proof}
  Let $\eps \in E_2$ be a unit; then $(\eps/\fP_1') = (\eps'/\fP_1)$, hence
  $(\eps/\fP_1)(\eps/\fP_1') = (\eps\eps'/\fP_1) = (N\eps/\fp_1) = +1$
  since $\fp_1$ is primary.
\end{proof}

We are therefore allowed to write $(E_2/\fp_1)$ instead of $(E_2/\fP_1)$.

\begin{prop}
  We have $(E_1/\fp_2) = (E_2/\fp_1)$, and there are two cases:
  \begin{itemize}
  \item $(E_1/\fp_2) = (E_2/\fp_1) = -1$: Then $h \equiv 2 \bmod 4$
    and $(E_F:NE_k) = 2$.
  \item $(E_1/\fp_2) = (E_2/\fp_1) = +1$: Then $4 \mid h(k)$.
  \end{itemize} 
\end{prop}

\begin{proof}
  If $(E_2/\fP_1) = +1$, then $\fP_1$ is primary, hence $\fP_1 = (\Pi_1)$
  for some primary element $\Pi_1$. Thus te extensions
  $k_2(\sqrt{\Pi_1}\,)/k_1$ and  $k_2(\sqrt{\Pi_1'}\,)/k_1$ are exactly
  ramified at $\fP_1$ and $\fP_1'$, respectively, hence $L = K(\sqrt{\Pi_1}\,)$
  is an unramified quadratic extension of $K$. This implies in particular
  that $h(K)$ is even and that $h(k)$ is divisible by $4$.

  Since $\Gal(L/F) \simeq D_4$, the extension $L/k_1$ is elementary abelian.
  Therefore there exist quadratic extensions $K_1/k_1$ and $K_1'/k_1$
  that are necessarily ramified at $\fP_2$ and $\fP_2'$, respectively.
  This implies that $\fP_2$ is primary, hence $(E_1/\fP_2) = +1$.

  The prime ideal $\fP_1$ in $k$ above $\fp_1$ splits in the Hilbert
  class field $K$ of $k$, hence it must be principal. This implies
  that $\Am_\st(k/F) = 1$, which in turn shows that $(E_F:NE_k) = 2$.
\end{proof}

This completes the proof of Thm.~\ref{Tmain}.
In the special case $F = \Q$ we recover Thm.~\ref{TSRC}. We will now give
two more examples of explicit versions of Thm.~\ref{Tmain}:

\begin{thm}\label{SRLi}
Let $\pi = a + bi$ and $\rho = c + di$ be primes in $\Z[i]$
with $\pi \equiv \rho \equiv 1 \bmod 4$; then $p = N\pi$ and 
$q = N\rho$ are primes congruent to $1 \bmod 8$. 
Let $h$ denote the $2$-class number of $K = k(\sqrt{\pi\rho}\,)$ and
$\eta$ its fundamental unit, and let $\eps_\pi$ and $\eps_\rho$ denote
the fundamental units of $\Q(i,\sqrt{\pi}\,)$ and $\Q(i,\sqrt{\rho}\,)$,
respectively.
\begin{enumerate}
\item If $[\pi/\rho] = -1$, then $h = 2$ and $N\eta = \pm i$.
\item If $[\pi/\rho] = +1$, then we have Scholz's reciprocity law
      $$ \Big(\frac{\eps_\pi}\rho\Big) = \Big(\frac{\eps_\rho}\pi\Big), $$
      and there are the following cases:
      \begin{enumerate} 
\item If $(\frac{\eps_\pi}\rho) = -1$, then 
      $h \equiv 2 \bmod 4$ and $N\eta = \pm 1$.
\item If $(\frac{\eps_\pi}\rho) = +1$, then $4 \mid h$.  
      \end{enumerate}      
\end{enumerate}
\end{thm}

In $\Q(\sqrt{-2}\,)$, the theorem becomes

\begin{thm}\label{SRL-2}
  Let $p \equiv q \equiv 1 \bmod 8$ be primes, and let
  $\pi = a + b\sqrt{-2}$ and $\rho = c + d\sqrt{-2}$ be primes
  in $\Z[i]$ with $a+b \equiv c+d \equiv 1 \bmod 4$ and $p = N\pi$
  and $q = N\rho$. Let $h$ denote the $2$-class number of
  $K = k(\sqrt{\pi\rho}\,)$ and $\eta$ its fundamental unit, and
  let $\eps_\pi$ and $\eps_\rho$ denote the fundamental units of
  $\Q(\sqrt{-2},\sqrt{\pi}\,)$ and $\Q(\sqrt{-2},\sqrt{\rho}\,)$,
  respectively.  
\begin{enumerate}
\item If $[\pi/\rho] = -1$, then $h = 2$ and $N\eta = -1$.
\item If $[\pi/\rho] = +1$, then we have Scholz's reciprocity law
  $$ \Big(\frac{\eps_\pi}\rho\Big) = \Big(\frac{\eps_\rho}\pi\Big), $$
  and there are the following cases:
   \begin{enumerate} 
   \item If $(\frac{\eps_\pi}\rho) = -1$, then $h \equiv 2 \bmod 4$
     and $N\eta = 1$.
\item If $(\frac{\eps_\pi}\rho) = +1$, then $4 \mid h$.
   \end{enumerate}   
\end{enumerate}
\end{thm}

\section{The Strong Form of Scholz's Reciprocity Law}
The classical version of Scholz's theorem is more precise than the theorem
we have proved above. In the case $F = \Q$, the reciprocity law given in
Thm.~\ref{TSRC} can be improved to
$$ \Big(\frac{\eps_p}q\Big) = \Big(\frac{\eps_q}p\Big)
   = \Big(\frac pq\Big)_4 \Big(\frac qp\Big)_4, $$
and the claim that $4 \mid h$ if $(\eps_p/q) = +1$ can be made more
precise as follows:
\begin{itemize}
\item If $(p/q)_4 = (q/p)_4 = -1$, then $h \equiv 4 \bmod 8$ and $N\eps = -1$.
\item If $(p/q)_4 = (q/p)_4 = +1$, then $8 mid h^+$, i.e., either
  $h \equiv 4 \bmod 8$ and $N\eps = +1$, or $8 \mid h$.
\end{itemize}
The theorems \ref{SRLi} and \ref{SRL-2} can be similarly improved (this will
be taken care of elsewhere), and I believe that this can be done in general
number fields with odd class numbers. In the following I will explain how
I think these results generalize.

\begin{cor}[Scholz]\label{Lhodd}
Let $F$ be a field with odd class number, and assume that 
$K/F$ is a quadratic extension in which exactly one prime $\fp$ 
ramifies. Then $K$ has odd class number, and every unit in $E_F$ 
is the norm of a unit from $E_K$.
\end{cor}

In fact, the ambiguous class number formula applied to the quadratic
extension $k/F$ shows that $\Am_2(k/F) = \frac{1}{(E_F:E_F \cap NK^\times)}$,
which shows that $\Am_2(k/F) = 1$ and that $k$ has odd class number $h$. But
then there is no nontrivial strongly ambiguous ideal class, hence
$\Am_\st(k/F) = \frac{h}{(E_F:NE_k)} = 1$, and this shows that $E_F = NE_k$.

We also need the following result:

\begin{prop}\label{PuN}
  If $k$ is a number field with cyclic $2$-class group and if $K/k$
  is a quadratic unramified extension, then $E_k = NE_K$.
\end{prop}

This is also a consequence of the ambiguous class number formula.
Since $\Cl_2(k)$ is cyclic, the Hilbert $2$-class field $k^1$ has
odd class number. The ambiguous class number formula applied
to $K^1/K$ tells us that every unit in $K$ is the norm of a unit
from $K^1$. But then it is also a norm from every subextension.

We can also prove this more directly: since ideals with order $2$ in the
ideal class capitulate in $K/k$ and since $K/k$ is unramified, there
cannot exist nontrivial strongly ambiguous ideal classes of order $2$
in $K$, and now the ambiguous ideal class number formula applied to $K/k$
shows that every unit in $E_k$ is the norm of a unit from $E_K$.

Now assume that there are two prime ideals $\fp_1$ and $\fp_2$
ramified in $k/F$, and that they are primary. Define characters $X_j$ 
($j = 1, 2$) on $E_F$ by first writing a unit $\eps \in E_F$ as the norm 
of units $\eta_j \in E_j$ (this can be done by Lemma \ref{Lhodd} below)
and then putting $X_1(\eps) = [\eta_1/\fP_2]$ and 
$X_2(\eps) = [\eta_2/\fP_1]$. This is well defined:

\begin{lem}
The symbol $[\eta/\fP_2]$ only depends on $N_1 \eta$, that is, 
units with the same norm have the same residuacity modulo $\fP_2$. 
\end{lem}

\begin{proof}
Since the norm is multiplicative, it suffices 
to show that $[\eta/\fP_2] = +1$ for all units $\eta \in E_1$
with $N\eta = +1$. But since $\eta = \eps^{\sigma-1}$ for some 
unit $\eps \in E_1$, we find
$[\eta/\fP_2] = [\eps/\fP_2] [\eps^\sigma/\fP_2] = (N\eps/\fp_2) = +1$
since $\fp_2$ is primary.
\end{proof}

Now we can prove:

\begin{prop}
For a number field $F$ with odd class number, let $\fp_1$ and
$\fp_2$ be odd primary prime ideals. Then the characters $X_j$ 
are well defined, and we have $X_1 = X_2 =:X$. Moreover, $X(\eps) = 1$
for all units $\eps \in E_F$ that are norms from $E_k$.
\end{prop}

\begin{proof}

  Assume that $\eps$ is a norm of a unit from $E_k$. We claim that
  $X_1(\eps) = X_2(\eps) = +1$. Since $\Cl_2(k)$ is cyclic and $K/k$
  is an unramified quadratic extension, we know from Proposition \ref{PuN}
  that $E_k = N_{K/k} E_K$. Thus $\eps$ is a norm of a unit from $E_K$.
  Let $\eta_1$ be its relative norm to $k_1$; since $\fP_2$ is ramified
  in $K/k_1$, we conclude that $+1 = [\eta_1/\fP_2] = X_1(\eps)$; exchanging
  the roles of $k_1$ and $k_2$ we find that $X_2(\eps) = +1$.  

  Since $\prod e(\fp) = 4$ in our case, the ambiguous class number
  formula tells us that $(E_F:NE_k) \le 2$.
  If this index is $1$, we are already done; thus assume that
  $(E_F:NE_k) = 2$, and let $\eps NE_k \in E_F/NE_k$ be the class which
  contains the units that are not norms from $E_k$. If our claim is wrong,
  then we must have $X_1(\eps) = - X_2(\eps)$. Assume without loss of
  generality that $X_1(\eps) = +1$. This implies that $(E_1/\fP_2) = +1$
  and $(E_2/\fP_1) = -1$, contradicting Theorem \ref{Tmain}. 
\end{proof}

\begin{cor}
  If $F \ne \Q(i)$ is a complex quadratic number field, then $-1$
  generates the unit group, and we have $X(-1) = (E_1/\fp_2) = (E_2/\fp_1)$.
\end{cor}

\begin{cor}
  We have $X(\eps) = +1$ for every unit $\eps \in E_F$ if and only if
  $(E_1/\fp_2) = (E_2/\fp_1) = +1$.
\end{cor}

We now observe that the rational biquadratic residue symbol $(\pi_1/\pi_2)_4$
is well defined for primary elements $\pi_1$, $\pi_2$ with $(\pi_1/\pi_2) = +1$.
In fact, if $\pi_1 \equiv \alpha^2 \bmod \pi_2$, then $\alpha$ is determined
modulo $\pi_2$ up to sign, and since $\pi_2$ is primary, we have
$(-1/\pi_2) = +1$; in particular,  $(\pi_1/\pi_2)_4 = (\alpha/\pi_2)$
does not depend on the choice of the sign of $\alpha$.

The only instances where $X(\eps)$ has been determined so far are the following:
\begin{itemize}
\item If $F = \Q$, then the primes $p$ and $q$ are primary if and only
  if they are positive and satisfy $p \equiv q \equiv 1 \bmod 4$; in this
  case we have
  $X(-1) = \big(\frac pq\big)_4\big(\frac qp\big)_4$.
\item If $F = \Q(\sqrt{-2}\,)$, then the primes $\pi_1 = a + b\sqrt{-2}$
  and $\pi_2 = c + d \sqrt{-2}$ are primary if
  $a+b \equiv c+d \equiv 1 \bmod 4$, and we have
  $X(-1) = \big(\frac{\pi_1}{\pi_2}\big)_4\big(\frac{\pi_2}{\pi_1}\big)_4$.
\item If $F = \Q(i)$,  the primes $\pi_1 = a + b\sqrt{-2}$
  and $\pi_2 = c + d \sqrt{-2}$ are primary if $b \equiv d \equiv 0 \bmod 4$,
  and we have $X(-1) = X(i)^2 = +1$ and
  $X(i) = \big(\frac 2p\big)_4\big(\frac 2q\big)_4$.
\end{itemize}
The proofs of the criteria for $X(i)$ in $\Z[i]$ and $X(-1)$ in
$\Z[\sqrt{-2}\,]$ will appear elsewhere. I do not know whether
$X(-1) = \big(\frac{\pi_1}{\pi_2}\big)_4\big(\frac{\pi_2}{\pi_1}\big)_4$
holds in every number field with odd class number.

\end{document}